\documentclass[11pt]{amsart}

\usepackage{graphicx}

\usepackage{amssymb,amsmath}

\usepackage{amsrefs}

\usepackage{latexsym}

\usepackage{epsfig}

\usepackage{color}

\usepackage{esint}

\newtheorem{theorem}{Theorem}[section]
\newtheorem{lemma}[theorem]{Lemma}

\newtheorem{cor}[theorem]{Corollary}

\theoremstyle{definition}
\newtheorem{de}[theorem]{Definition}
\newtheorem{example}[theorem]{Example}

\theoremstyle{definition}
\newtheorem{remark}[theorem]{Remark}

\theoremstyle{proposition}
\newtheorem{proposition}[theorem]{Proposition}

\numberwithin{equation}{section}

\def\eq#1{(\ref{#1})}

\def\o{\Omega}

\def\RR{{\mathbb R} }

\def\ri{\rightarrow}

\def\bom{\bar\Omega}

\def\om{\Omega}
\def\R{\mathbb{R}}
\def\sub{\underline}
\def\super{\overline}

\begin{document}

\title[Entire Large Solutions for Semilinear Elliptic Equations ]
{Entire Large Solutions for Semilinear Elliptic Equations }

\author[L. Dupaigne]{Louis Dupaigne}

\address{LAMFA, UMR CNRS 6140, Universit\'e de Picardie Jules Verne, 33 rue St Leu, 80039, Amiens Cedex, France}

\email{louis.dupaigne@math.cnrs.fr}

\author[M. Ghergu]{Marius Ghergu}

\address{School of Mathematical Sciences, University College Dublin, Belfield, Dublin 4, Ireland}

\email{marius.ghergu@ucd.ie}

\author[O. Goubet]{Olivier Goubet}

\address{LAMFA, UMR CNRS 7352, Universit\'e de Picardie Jules Verne, 33 rue St Leu, 80039, Amiens Cedex, France}

\email{olivier.goubet@u-picardie.fr}

\author[G. Warnault]{Guillaume Warnault}

\address{LMA, UMR CNRS 5142, Universit\'e de Pau et des Pays de l'Adour, Avenue de l'Universit\'e, BP 1155, 64013 Pau Cedex, France}

\email{guillaume.warnault@univ-pau.fr}

\subjclass[2010]{Primary 35B40; 35J75; Secondary 35J25; 35B51}


\keywords{Entire (large) solution; radial symmetry; asymptotic behavior, uniqueness}

\begin{abstract}
We analyze the semilinear elliptic equation $\Delta u=\rho(x) f(u)$, $u>0$ in $\R^D$ $(D\ge3)$,
with a particular emphasis put on the qualitative study of entire large solutions, that is, solutions $u$ such that $\lim_{\vert x\vert\to+\infty}u(x)=+\infty$.
Assuming that $f$ satisfies the Keller-Osserman growth assumption and
that $\rho$ decays at infinity in a suitable sense, we prove the existence of entire large solutions.
We then discuss the more delicate questions of asymptotic behavior at infinity, uniqueness and symmetry of solutions.
\end{abstract}

\maketitle

\tableofcontents

\section{Introduction}

We investigate the semilinear elliptic equation
\begin{equation}\label{els}
\Delta u=\rho(x) f(u)\,,\;\;
u>0 \qquad\text{in $\R^D$ $(D\ge3)$,}
\end{equation}
where $\rho, f$ are {\it positive} quantities, satisfying general growth assumptions to be specified in the following.
The above equation appears naturally in a number of interesting contexts which we recall now.

The link between semilinear elliptic equations and conformal geometry has been known for a long time (see e.g. the seminal work of H. Yamabe \cite{yamabe}, as well as as the lecture notes of E. Hebey \cite{hebey}): when $f(u)=u^\frac{D+2}{D-2}$, the solvability of \eqref{els} is equivalent to the existence of a conformal metric on the Euclidean space $\R^D$, with prescribed scalar curvature $K=-\rho$. Up to a dilation, the corresponding conformal factor is the quantity $\varphi=u^{\frac4{D-2}}$. For the study of this equation, see e.g. W.-M. Ni \cite{ni}, Y. Li and W.-M. Ni \cite{li-ni}, and K.-S. Cheng and W.-M. Ni \cite{cheng-ni}.

It is also known that properties of random systems of branching particles are related to semilinear elliptic equations of the form \eqref{els}, when $f(u)=u^p$, $1<p\le2$. See  the pioneering work \cite{dynkin} of E.B. Dynkin, as well as the review paper of J.F. Le Gall \cite{legall}. When $\rho$ is bounded, the parabolic version of \eqref{els} is the log-Laplace equation of a measure-valued branching process $(X_{t})$, known as a catalytic super-Brownian motion. Several properties of $(X_{t})$ can be derived from the study of \eqref{els}. For example, the process has compact global support (that is, the closure of the union of the supports of all measures $X_{t}$, $t\ge0$ is almost surely compact) if and only if \eqref{els} fails to be solvable.
See J. Engl\"ander and R. G. Pinsky \cite{ep} and Y.-X. Ren \cite{ren}.

From the PDE perspective, the classification of solutions to \eqref{els} (in particular, questions of existence, uniqueness, radial symmetry, and asymptotic behavior at infinity) is of interest, also because it can provide information, such as {\it a priori} estimates on solutions to the same equation, posed in an arbitrary proper domain of Euclidean space. See the seminal paper of J.B.  Keller \cite{K} and the introduction of A. Olofsson in \cite{olofsson} for the case $\rho\equiv1$,  as well as S. Taliaferro's work \cite{taliaferro2, taliaferro} and the references therein, for more general situations.

Finally, from the point of view of exponential asymptotics, {\it entire large solutions} (ELS, for short), that is, solutions such that
$$
\lim_{\vert x\vert\to+\infty}u(x)=+\infty,
$$
provide an interesting example where the function $u$ is no better in general than the Borel sum of a factorially divergent series, while the inverse mapping $r=r(u)$ of a radial ELS $u=u(r)$ turns out to be, at least in some cases,  the sum of a convergent but {\it abstract} asymptotic expansion. To illustrate this, consider the case where $\rho(x)=\vert x\vert^{2-2D}$ and $f(u)=u(\ln u)^4$. If $u=u(r)$ is a radial ELS, then $v(t)=u(r)$ with $t=r^{2-D}$ solves the autonomous ODE
$$
v'' = \frac{1}{(D-2)^2}f(v)
$$
and blows up at $t=0$. A formal calculation leads to the asymptotic expansion
$$
v\sim e^{1/t}\sum a_{k}t^k,
$$
where the coefficients $a_{k}$ exhibit factorial divergence. It can be proved that the above series is Borel summable. Furthermore, if $\tilde v$ denotes its Borel sum, $v-\tilde v$ is exponentially small. Instead, if one tries to expand $t$ as a function of $v$, one recovers a {\it convergent} power series in the unknown $z=(\ln v)^{-1}$.

A remarkable fact is that such a convergent asymptotic expansion can be obtained for any nonlinearity $f$ satisfying the Keller-Osserman growth condition (see $(KO)$ below). Each term in the expansion is ``abstract'' i.e. computed in terms of iterated antiderivatives of $f$. See the work \cite{cd} by O. Costin and one of the authors for a similar situation.

\

\noindent Now, let us turn to the structural assumptions made on the data $f$ and $\rho$.
\begin{itemize}
\item
First, we restrict our attention to the case where $\rho, f>0$: it is well known that the analysis of the PDE \eqref{els} is radically different under different sign assumptions on the data and we shall not elaborate on this restriction, apart from saying that, for some of our results, it suffices to assume that $f$ is positive only at infinity, in the following sense (due to H. Brezis, see \cite{cd}):
\begin{equation*} 
\exists\; a\in\R^+\quad\text{s.t.}\quad f(a)>0\quad\text{and} \quad f(t)\ge0\quad\text{for $t>a$},
\eqno{(P_{f})}
\end{equation*}
and that $\rho$ can vanish only in the following sense (due to A.V. Lair, see \cite{lair}):
\begin{center}
$\rho\ge0$ and for all $x_{0}\in\R^D$ such that $\rho(x_{0})=0$, there exists a bounded domain $\Omega\subset\R^D$ containing $x_{0}$ such that $\rho\vert_{\partial\Omega}>0$.\hfill{$(P_{\rho})$}
\end{center}
\item Next, we assume that $f$ is superlinear in the sense that
$$\int^{+\infty}
\frac{ds}{\sqrt{F(s)}}<+\infty,
\eqno(KO)$$
where $F(s)=\int_a^s f(t)dt$. This assumption, first introduced by J.B. Keller \cite{K} and R. Osserman \cite{O}, is structural: take for example the simpler case where $\rho\equiv1$. If $f\ge0$ satisfies $(KO)$, then \eqref{els} has no nontrivial solution (see \cite{K,O,d1}). In turn, if $f~\ge~0$ fails to fulfill $(KO)$, problem \eqref{els} has infinitely many (radial, entire large) solutions (see \cite{K,O}). Furthermore, at least in the specific case where $f(u)=u^q$ for certain values of $q\in(0,1]$ (and so, again, $(KO)$ fails), the equation also admits nonradial ELS (see \cite{taliaferro, bvg}). So, at present, $(KO)$ seems to be a necessary assumption in order to classify all solutions to the equation.

\item Whenever the Keller-Osserman condition $(KO)$ holds, it is natural to request that $\rho$ decays fast at infinity, in the sense that there exists a solution to
\begin{equation} \label{condH}
\left\{
\begin{aligned}
-\Delta U&=\rho(x)\qquad\text{in $\R^D$, $D\ge3$},\\
\lim_{\vert x\vert\to+\infty}U(x)&=0.
\end{aligned}
\right.
\end{equation}
Using the results of Appendix A in H. Brezis, S. Kamin \cite{bk}, the solvability of \eqref{condH} is equivalent to
$$
\lim_{\vert x\vert\to+\infty}\int_{\R^D}\vert x-y\vert^{2-D}\rho(y)\;dy=0.
\eqno(H_\rho)
$$
When $\rho$ is radial, this simplifies to
\begin{equation*}
\int_0^{+\infty} r\rho(r)dr<+\infty.
\end{equation*}
As we shall see, assumptions $(KO)$ and $(H_\rho)$ turn out to be sufficient for the existence of an ELS to \eqref{els} (see also D. Ye and F. Zhou \cite{yz, dong}, for a proof under the additional assumption that $f$ is increasing). In fact, if e.g. $f(u)=u^p$, $p>1$, and $\rho$ is radial, $(H_\rho)$ is also necessary for the existence of an ELS as shown in \cite{lair, taliaferro}.
\item Finally, to avoid technicalities, we assume that $f$ and $\rho$ are $C^1$ regular, and that $f(0)=0$.
\end{itemize}

\section{Main results}
We are now in a position to state our main results. We begin with the existence theory.
\begin{theorem}\label{existence} {\rm (Existence of bounded and large solutions)}  Assume that $f\not\equiv0$ is a $C^1$ function such that $f(0)=0$, $f(t)>0$ for $t>0$,  and $(KO)$ holds. Assume that $\rho>0$ is a $C^1$ function  satisfying $(H_\rho)$. Then, for every $\beta\in(0,+\infty]$, there exists a minimal solution to \eqref{els} such that
\begin{equation} \label{limitbeta}
\lim_{\vert x\vert\to+\infty}u(x)=\beta.
\end{equation}
\end{theorem}
In the above theorem, we used the following definition.
\begin{de}For every $\beta\in(0,+\infty]$,
$u$ is the minimal solution of \eqref{els} satisfying \eqref{limitbeta}, if for any supersolution  $\super u>0$ of \eqref{els} such that
$$
\liminf_{\vert x\vert\to+\infty}\super u(x)\ge\beta,
$$
we have
$$0<u\le\super u \qquad\text{in $\R^D$}.$$
\end{de}

\begin{remark}\label{remark existence} Theorem \ref{existence} remains valid under the weaker sign assumption $(P_{\rho})$ on $\rho$. Also, under the weaker sign condition $(P_{f})$ on $f$,  Theorem \ref{existence} remains valid for all $\beta\in[a,+\infty]$ where the constant $a$ is defined in $(P_{f})$.
\end{remark}

\begin{remark}
The existence of bounded solutions has been investigated by many authors. See in particular \cite{yz}, where nonlinearities failing the $(KO)$ condition are also considered. In the same paper, the authors construct large solutions (i.e. solutions satisfying \eqref{limitbeta} with $\beta=+\infty$) under the $(KO)$ condition and under the additional assumptions that $\rho$ is positive everywhere, and that $f$ is nonnegative and nondecreasing. By Theorem \ref{existence} and Remark \ref{remark existence}, the positivity assumptions can be relaxed, while the monotonicity assumption can be simply removed. In fact, we believe (and give evidence later on) that $(KO)$ is the good assumption to classify all solutions to our semilinear problem, without assuming that
$f$ is nondecreasing. In addition, the existence of a {\rm minimal}  solution satisfying \eq{limitbeta} (in particular the existence of a minimal ELS) is new.
\end{remark}

Our next observation is that all bounded solutions to \eqref{els} must have a limit at infinity.
\begin{theorem}\label{prop limit}{\rm  (Any bounded solution has a limit)} Make the same assumptions as in Theorem \ref{existence}. Let $u$ be a bounded solution of \eqref{els}. Then, \eqref{limitbeta} holds for some $\beta\in(0,+\infty)$.
\end{theorem}

\begin{remark}
The above theorem is essentially known:  see in particular \cite{yz} for the case where $f$ is nondecreasing.
\end{remark}

Under some mild (but technical) assumptions on $\rho$ and $f$ a similar result holds for unbounded solutions. More precisely we have:
\begin{theorem}\label{tunbounded}{\rm  (Any unbounded solution is an ELS)}
Make the same assumptions as in Theorem \ref{existence}. Assume in addition that 
\begin{enumerate}
\item[(i)] there exist $c\in (0,1)$ and $\alpha>2$ such that
\begin{equation}\label{rhoc}
c|x|^{-\alpha}\leq \rho(x)\leq \frac{1}{c}|x|^{-\alpha}\quad\mbox{for all }|x|>1;
\end{equation}
\item[(ii)] there exists $M>0$ such that the mapping $u\longmapsto f(u)/u$ is nondecreasing in $I=[M,+\infty)$ and there exists $C>0$ such that
\begin{equation}\label{fiddgr}
\frac{f(u)}{u}\leq \frac{C}{\Phi^2(u)}\quad \mbox{ for all $u\in I$,}
\end{equation}

where
\begin{equation}\label{fiddgr2}
\Phi(u)=\int_u^{+\infty}\frac{dt}{\sqrt{F(t)-F(u)}}.
\end{equation}
\end{enumerate}

Let $u$ be an unbounded solution of \eqref{els}. Then, \eqref{limitbeta} holds for $\beta=+\infty$.
\end{theorem}

\begin{remark}\label{rko}
(i) Since $f$ satisfies $(KO)$ it is easily seen that $\Phi$  is well defined. Furthermore, since $f$ is increasing in $I$, $\Phi$ is decreasing and bijective in $I$(see Lemma \ref{propPhi} below).

(ii) Conditions \eq{rhoc} and \eq{fiddgr} are motivated by the results in \cite{cli} and \cite{cheng-ni}. Inequality \eq{fiddgr} is satisfied by nonlinearities $f$ with either power type or exponential growth. Indeed if $f(u)\eqsim u^p$, $p>1$, then $\Phi(u)\eqsim u^{(1-p)/2}$ and if $f(u)\eqsim e^u$ then $\Phi(u)\eqsim e^{-u/2}$, so in both cases \eq{fiddgr} holds.
\end{remark}

Next, we point out that for a fixed $\beta\in(0,+\infty]$ there may be many solutions of \eqref{els} that satisfy \eqref{limitbeta}. In particular, there is in general no maximal ELS:

\begin{remark}\label{corsandwich0}
Assume that $\rho>0$ is a $C^1$ function satisfying $(H_\rho)$. Assume that $f$ is a $C^1$ function, satisfying $(P_{f})$ and  $(KO)$.    Assume in addition that $f$ vanishes infinitely many times near infinity, i.e. there exists a sequence $\{t_k\}\subset\R^+$ such that
$f(t_k)=0$ for all $k\geq 1$ and { $\lim_{k\ri +\infty}t_k=+\infty$.}
Then, \eq{els} has infinitely many ELS but no maximal ELS.

Indeed, let $f_k(t)=f(t+t_k)$, $t\geq 0$, $k\geq 1$. Then, by Remark \ref{remark existence},  there exists an ELS $v_k>0$ of $\Delta v_k=\rho(x)f_k(v_k)$ in $\RR^D$. Set  $u_k=v_k+t_k$. Then, $u_k$ is an ELS of \eq{els} and $u_k\geq t_k$. Since $\{t_k\}$ is unbounded, infinitely many $u_{k}$'s are distinct, and if there existed a maximal ELS of \eq{els}, say $V$, then we would have $V\ge u_{k}\ge t_{k}$. Letting $k\to+\infty$ this yields a contradiction. \end{remark}



\begin{example}\label{example}
The nonlinearity $f(u)=u^2(1+\cos u)$ satisfies all the requested assumptions (see the work \cite{ddgr} by S. Dumont, V. Radulescu and two of the authors for the validity of the Keller-Osserman condition $(KO)$ in this specific case).
\end{example}

\medskip

In contrast to the above result, when $f$ is nondecreasing and $\beta<+\infty$, it easily follows from the maximum principle that the solution to \eqref{els}-\eqref{limitbeta} is unique. Does this remain true for ELS?

We deal first with the case where $\rho(x)= \vert x\vert^{-\alpha}$ for large $\vert x \vert$.
\begin{theorem}\label{uniq} Make the same assumptions as in Theorem \ref{existence}.
Suppose also that 
for some $\alpha>2$,
\begin{equation}\label{cond}
\rho(x)= \vert x\vert^{-\alpha} \mbox{ for } \vert x\vert\geq 1.
\end{equation}
Then, given two ELS $u_{1},u_{2}$, there holds
$$
\lim_{\vert x\vert\to+\infty} [u_{1}(x)-u_{2}(x)] = 0.
$$
\end{theorem}
As an immediate corollary, we find:
\begin{cor}Make the same assumptions as in Theorem \ref{existence}.
Suppose also that $f$ is nondecreasing and that \eqref{cond} holds for some $\alpha>2$. Then, there exists exactly one ELS to \eqref{els}.
\end{cor}
Applying the moving-plane procedure as in \cite{ni-li} for the case $\beta<+\infty$ and as in \cite{taliaferro} for the case $\beta=+\infty$, we also have immediately:
\begin{cor}\label{cor mp} Fix $\beta\in(0,+\infty]$. Make the same assumptions as in Theorem \ref{existence}.
Suppose also that $f$ is nondecreasing on some interval $[M,\beta)$, and that $\rho$ is a radially decreasing function such that \eqref{cond} holds for some $\alpha>2$. Then, every solution to \eqref{els}-\eqref{limitbeta} is radial.
\end{cor}

\begin{remark}
It would be interesting to know whether Corollary \ref{cor mp} remains true for oscillating nonlinearities such as the one in Example \ref{example}.
\end{remark}
Next, we are able to extend the previous results to the case where $\rho$ is a perturbation of the model case $\rho(x)=\vert x\vert^{-\alpha}$, $\alpha=2D-2$.
\begin{theorem}\label{uniqueness}
Make the same assumptions as in Theorem \ref{existence}. Suppose also that there exists a constant $C>0$ such that
\begin{equation} \label{cond2}
\rho(x)=\vert x\vert^{2-2D}(1+\sigma(\vert x\vert)),\quad\text{where}\quad \left\vert \frac{d\sigma}{dr}(\vert x\vert)\right\vert\le C\vert x\vert^{1-D}\quad\text{for $\vert x\vert\ge1$}.
\end{equation}
Then, given two ELS $u_{1},u_{2}$, there holds
$$
\lim_{\vert x\vert\to+\infty} [u_{1}(x)-u_{2}(x)] = 0.
$$
\end{theorem}
These are our best results without making any assumption on the nonlinearity $f$, set aside the structural Keller-Osserman condition $(KO)$.
In the next set of results, we investigate the question of uniqueness for more general potentials $\rho(x)$ under an extra convexity assumption on the nonlinearity $f$. We begin with the case where $\rho$ is radial.

\begin{theorem}\label{coro3}
Make the same assumptions as in Theorem \ref{existence}. Assume in addition that
\begin{enumerate}
\item[(i)] $f$ is nondecreasing,
\item[(ii)] $\sqrt F$ is convex on $[M,+\infty)$ for some $M>0$,
\item[(iii)] $\rho$ is radially symmetric, and
\item[(iv)] $r^{2D-2}\rho(r)$ is nondecreasing on $[R,+\infty)$, for some $R>0$.
\end{enumerate}
Then there exists a unique ELS of \eq{els}.
\end{theorem}
It is possible to extend the previous result to nonradial $\rho$, provided some extra information on the mean curvature of its level sets is available.
\begin{theorem}\label{ELSconvex}
Make the same assumptions as in Theorem \ref{existence}. Assume in addition that
\begin{enumerate}
\item[(i)] $f$ is nondecreasing;
\item[(ii)] $\sqrt F$ is convex on $[M,+\infty)$ for some $M>0$;
\item[(iii)] $\lim_{\vert x\vert\rightarrow+\infty}\rho(x)=0$;
\item[(iv)]  $\sqrt{\rho}$ is superharmonic in $\RR^D\setminus B_{R}$ for some $R>0$;
\item[(v)] $\rho\in C^{D+1}(\R^D)$ and for a sequence of regular values $\rho_{n}\to 0^+$,
\begin{equation}\label{meanc}
2(D-1)H_{n}\geq\frac{|\nabla \rho|}{\rho}\quad\mbox{ on }[\rho=\rho_{n}],
\end{equation}
where $H_{n}$ denotes the mean curvature of the level set $[\rho=\rho_{n}]$ (with the usual sign convention that $H_{n}\ge0$ whenever $[\rho>\rho_{n}]$ is convex).
\end{enumerate}
Then, there exists a unique ELS of \eq{els}.
\end{theorem}

\begin{remark}
(i) Since $\rho\in C^{D+1}$, it follows from the Morse-Sard lemma that almost all values of $\rho$ are regular and that the corresponding level sets are smooth enough to define their mean curvature. Since $\rho(x)\to0$ as $\vert x\vert\to+\infty$, the level sets are compact, nested, and their union covers $\RR^D$.

(ii) When $\rho$ is radial, \eqref{meanc} reduces to  $\frac d{dr}\left(r^{2D-2}\rho(r)\right)\ge0$ for $r=r_{n}\to+\infty$.
\end{remark}

\begin{example}
Let us try to understand conditions (iv) and (v) in Theorem \ref{ELSconvex} on a simple example: when the level sets of $\rho$ are ellipsoids. Fix $\alpha>2$ and $a\in(0,1)$. For $x=(x_{1},x')\in\R^{D-1}\times\R$, let
$$
v(x) = \left(\frac{x_{1}}{a}\right)^2 + \vert x'\vert^2\quad\text{and}\quad \rho(x)=v(x)^{-\alpha/2}.
$$
Then, $\rho$ and $v$ share the same level sets and by a direct computation, \eqref{meanc} holds if and only if
$$
\alpha \le a^2(2D-2),
$$
while (iv) holds if and only if
$$
(\alpha+2)\le a^2(2D-2).
$$
Under the latter condition, our theorem applies, that is, if $D\ge4$ and the ellipsoid is not too flat, then uniqueness holds.
\end{example}



The remaining part of the paper is organized as follows. In the next section we are concerned with the existence of solutions to \eq{els}, namely we prove Theorem \ref{existence}. In Section 4 we prove Theorem \ref{prop limit} and Theorem \ref{tunbounded} regarding the behavior at infinity of solutions to \eq{els}. In Section 5 we study the uniqueness of ELS to \eq{els} and prove Theorems \ref{uniq}, \ref{uniqueness}, \ref{coro3} and \ref{ELSconvex}. For the reader's convenience we recalled the most important results used in the proofs in Appendix C.

\section{Existence of solutions}

In this section, we prove Theorem \ref{existence}. The first step consists in constructing a subsolution:

\begin{proposition}\label{w}
Assume that $f$ is a $C^1$ function, satisfying $(P_{f})$ and $(KO)$, and such that $f(0)=0$. Assume that $\rho\ge0$ is a $C^1$ function with superlinear decay in the sense of $(H_\rho)$.
Then, for any $\beta\in(0, +\infty]$, there exists a 
function $w_{\beta}\in C^2(\R^D)$ such that
\begin{equation}\label{beta}
\left\{
\begin{aligned}
&\Delta w_{\beta}\geq \rho(x) f(w_{\beta})&&\quad\mbox{ in }\R^D,\\
&\lim_{|x|\to +\infty}w_{\beta}(x)=\beta.
\end{aligned}
\right.
\end{equation}
Moreover, $0< w_{\beta}<\beta$, the family $\{w_{\beta}\}_{\beta\in(0,+\infty]}$ is increasing in $\beta$, and $\lim_{\beta\ri+\infty}w_{\beta}(x)=w_{\infty}(x)$, for all $x\in\R^D$.
\end{proposition}
\begin{proof}
Let $\bar f\in C^1[0,+\infty)$ be an increasing function such that
$$
\bar f\geq f,\quad  \bar f(0)=0,\quad\mbox{ and }\quad \bar f>0 \mbox{ in } (0,+\infty).
$$
Since $f$ satisfies $(KO)$, so does $\bar f$. Next by \cite[Lemma 1]{lair} (see also \cite{dong}) we have
\begin{equation}\label{bar0}
\int^{+\infty} \frac{1}{\bar f(s)}ds<+\infty.
\end{equation}
Further, using $\bar f(0)=0$ and \eq{bar0}, we derive that for all $0<\beta\leq +\infty$ the mapping
$$(0,\beta)\ni t\mapsto \int_t^{\beta}\frac{ds}{\bar f(s)}\in (0,+\infty)$$
is bijective. Therefore, for any $\beta\in(0,+\infty]$,
there is a unique
$$w_{\beta}:\RR^D\rightarrow (0,\beta)$$ such that
\begin{equation}\label{beta1}
\int^\beta_{w_{\beta}(x)}\frac{ds}{\super f(s)}=U(x)\quad\mbox{ for all } x\in\RR^D,
\end{equation}
where $U$ is given by \eqref{condH}.
Clearly, $w_\beta$ is increasing with respect to $\beta$ and $\lim_{\beta\to+\infty}w_{\beta}(x)=w_{\infty}(x)$ for all $x\in\R^D$. 
Now,
$$
\nabla U(x)=-\frac{1}{\bar f(w_\beta(x))}\nabla w_\beta(x)\quad\mbox{ in }\R^D
$$
and
$$
\rho=-\Delta U=\frac{1}{\bar f(w_\beta)}\Delta w_\beta-\frac{\bar f'(w_\beta)}{\bar f^2(w_\beta)}|\nabla w_\beta|^2\leq \frac{1}{\bar f(w_\beta)}\Delta w_\beta\quad\mbox{ in }\R^D.$$
Hence, $w_\beta$ satisfies \eq{beta}.
\end{proof}
\subsection{ Proof of Theorem \ref{existence} and Remark \ref{remark existence}}
Let us start with the simpler case $\beta<+\infty$ (and $\beta\ge a$ if $f$ satisfies only $(P_{f})$). Observe that the functions $\sub u=w_{\beta}$ given by Proposition \ref{w}  and $\super u=\beta$ are respectively a sub and a supersolution to the problem
$$
\left\{
\begin{aligned}
\Delta  u&= \rho(x)f(u)&&\quad\mbox{ in } B_R, \\
u& = w_\beta &&\quad\mbox{ on } \partial B_R,
\end{aligned}
\right.
$$
where $R>0$. By Proposition A.1, the above problem has a minimal solution $u_{R}$ relative to $w_{\beta}$. In particular,
$$w_{\beta}\le u_{R}\le \beta.$$ By standard elliptic regularity, a sequence $\{u_{R_{n}}\}$ converges in $C^2_{loc}(\R^D)$ to a solution $u_{\beta}$ of \eqref{els} that satisfies \eqref{limitbeta}. It remains to prove that $u_{\beta}$ is minimal. By Proposition A.1, it suffices to prove that any supersolution $\super u$ of \eqref{els}-\eqref{limitbeta} verifies $\super u\ge w_{\beta}$. From the proof of Proposition \ref{w}, there exists an increasing function $\super f\ge f$ such that
\begin{equation}\label{eqbbb}
\Delta w_{\beta}\ge \rho(x)\super f(w_{\beta})\qquad\text{in $\R^D$,}
\end{equation}
while clearly $\super u$ satisfies the reverse inequality. Since $\super f$ is increasing, it follows from the maximum principle that $\super u\ge w_{\beta}$, as desired.

Now, let us turn to the remaining case $\beta=+\infty$.
For any $R>0$, $\sub u=w_{\infty}$ and $\super u= \| w_{\infty}\|_{L^\infty(B_{R})}$ are respectively a sub and a supersolution to the problem
$$
\left\{
\begin{aligned}
\Delta  u&= \rho(x)f(u)&&\quad\mbox{ in } B_R, \\
u& = w_\infty &&\quad\mbox{ on } \partial B_R.
\end{aligned}
\right.
$$
By Proposition A.1, the above problem has a minimal solution $u_R$ relative to $w_\infty$, and for all $R>R'>0$,
\begin{equation}\label{wwwR}
w_\infty\leq u_{R}\leq u_{R'} \quad\mbox{ in }B_R.
\end{equation}
 Let us prove that the family $\{u_{R}:{R\ge1}\}$ is uniformly bounded on compact sets of $\R^D$. To do so, it suffices to prove that given $x\in\R^D$, $\{u_{R}:R\ge1\}$ remains bounded in some neighborhood of $x$. If $\rho(x)>0$, there exists $r=r_{x}>0$ such that $m_{r}=\inf_{B(x,r)}\rho>0$.
By Theorem 1.3 in \cite{ddgr},  there exists $U_r$, the minimal solution (relative to $w_\infty$) to the problem
\begin{equation}\label{bbusR}
\left\{
\begin{aligned}
\Delta  U_r&= m_{r}f(U_r)&&\quad\mbox{ in }B(0,r), \\
 U_r& = +\infty &&\quad\mbox{ on } \partial B(0,r).
\end{aligned}
\right.
\end{equation}
By Proposition A.1, $u_{R}(y)\le U_{r}(y-x)$ for $y\in B(x,r)$ and $R\ge r$. In particular, $\{u_{R}\}$ remains uniformly bounded in the ball $B(x,r/2)$.  Assume now that $\rho(x)=0$. By the assumption $(P_{\rho})$, there exists a bounded domain $\Omega$ containing $x$ such that $\rho\vert_{\partial\Omega}>0$. Using again the barrier given by \eqref{bbusR} at every point of $\partial\Omega$, we deduce that $\{u_{R}\}$ remains uniformly bounded by some constant $K$ in a neighborhood of $\partial\Omega$. By the assumption $(P_{f})$, there exists $a\ge0$ such that $f(t)\ge0$ for $t\ge a$. In particular, for any $R\ge1$, $u_{R}$ is subharmonic in $\Omega\cap[u_{R}>a]$. It follows from the maximum principle that $u_{R}\le\max\{a,K\}$ in $\Omega$.
So the family $\{u_{R}:{R\ge1}\}$ is uniformly bounded on compact sets of $\R^D$ and satisfies \eqref{wwwR}. By elliptic regularity, as $R\to+\infty$, $u_R$ converges in $C^2_{loc}(\R^D)$ to a solution $u_{\infty}$ of \eqref{els} such that $u_{\infty}\ge w_{\infty}$. It follows that $u_{\infty}$ is an ELS. To show the minimality of $u_{\infty}$, take a supersolution $\super u$ such that $\lim_{\vert x\vert\to+\infty} \super u(x)=+\infty$. From \eqref{eqbbb} and the maximum principle, we infer that $\super u\ge w_{\beta}$ for all $\beta<+\infty$. Letting $\beta\to+\infty$, we deduce that $\super u\ge w_{\infty}$. By Proposition A.1, it easily follows that $\super u\ge u_{\infty}$, as desired.
\hfill\qed
\begin{remark} If $f$ is nondecreasing we can simply work with $f$ instead of $\bar f$ in the definition of $w_\infty$ given in Proposition \ref{w}.
In this case, from \eq{beta1} and the fact that any  ELS $u$ of \eq{els} satisfies $u\geq w_\infty$ we find the following implicit lower bound on the growth of $u$ at infinity
$$
\int_{u(x)}^{+\infty} \frac{ds}{f(s)}\leq U(x)\qquad\text{for all $x\in\R^D$},
$$
where $U$ is the solution to  \eqref{condH}.
\end{remark}

\section{All solutions have a limit at infinity}
In this section, we prove Theorems \ref{prop limit} and \ref{tunbounded}. We begin with the case of bounded solutions.

\subsection{Proof of Theorem \ref{prop limit}} Let $u$ be a bounded solution of \eq{els}.
 By Lemma B.1, the unique solution to \eqref{condH} is given by
$$
U(x)=c_{D}\int_{\R^D}\vert x-y\vert^{2-D}\rho(y)\;dy,
$$
where $c_{D}\vert x\vert^{2-D}$ is the fundamental solution of the Laplace operator. Since $u$ is bounded and $f$ is continuous, the function $V$ defined for all $x\in\R^D$ by
$$
V(x)=c_{D}\int_{\R^D}\vert x-y\vert^{2-D}\rho(y)f(u(y))\;dy,
$$
 satisfies
$$
\vert V\vert \le C U\qquad\mbox{in }\R^D,
$$
for some constant $C>0$. Since $\lim_{\vert x\vert\to+\infty}U(x)=0$, it follows from Lemma B.1 (applied to $h=\rho[f(u)+ \| f(u) \|_{L^\infty(\R^D)}]$) that $V$ solves
\begin{equation*}
\left\{
\begin{aligned}
-\Delta V &=\rho(x) f(u)&\quad\text{in $\R^D$,}\\
\lim_{\vert x\vert\to+\infty}V(x)&=0.&
\end{aligned}
\right.
\end{equation*}
Hence, $u+V$ is a bounded harmonic function in $\R^D$. By Liouville's theorem, $u+V$ must be equal to a constant $\beta$. Since $u>0$, we derive $\beta\ge0$. For $r>0$, let
$$
\super u(r)=\fint_{\partial B_{r}(0)}u\;d\sigma.
$$
Since $u$ is subharmonic, it follows that $\super u$ is a nondecreasing function of $r$. This implies that $\beta>0$, as requested.
\hfill\qed

Next, we deal with unbounded solutions to \eq{els}. Before proving Theorem \ref{tunbounded} we need two auxiliary results.
\begin{lemma}\label{sandwich} Make the same assumptions as in Theorem \ref{existence}. Suppose in addition that  $\rho$ is radial and nonincreasing. Then, for any function $u$ such that
$$
\Delta  u\ge \rho(|x|)f(u)\quad\mbox{ in } \R^D, \\
$$
there exists a radial function $\super u$ solving \eqref{els} and such
that $u\leq \super u$ in $\R^D$.
\end{lemma}
\proof Let $u$ satisfy the above differential inequality. We fix $R>0$ and let $N=N(u,R)\geq 1$ be such that $\max_{B_R}u<N$.
By Proposition A.1, for all $n\geq N$, there exists a minimal solution $u_R^n$ relative to $u$ of the problem
$$
\left\{
\begin{aligned}
\Delta  u_R^n&= \rho(|x|)f(u_R^n)&&\quad\mbox{ in } B_R, \\
u_R^n& = n &&\quad\mbox{ on } \partial B_R,
\end{aligned}
\right.
$$
such that $u\leq u^n_R<n$ in $B_R$.
Since $r\mapsto \rho(r)$ is nonincreasing, the Gidas-Ni-Nirenberg symmetry result (Theorem 1' in \cite{gnn}) implies that $u_R^n$ is radially symmetric.
Let $m_{R}=\inf_{B_{R}}\rho>0$ and $U=U_R$ be the minimal solution to relative to $u$ of
$$
\left\{
\begin{aligned}
\Delta  U&= m_{R}f(U)&&\quad\mbox{ in } B_R, \\
U& = +\infty &&\quad\mbox{ on } \partial B_R,
\end{aligned}
\right.
$$
Applying Proposition A.1, we have
\begin{equation}\label{mis0}
u\leq u_R^n\leq U_R\quad\mbox{ in }B_R.
\end{equation}
Applying further Proposition A.1, we find
\begin{equation}\label{mis1}
u_R^n\leq u_R^{n+1},\quad u^{n}_{R+1}\leq u_R^n\quad\mbox{ in }B_R, \mbox{ for all }\, n\geq N.
\end{equation}
Hence, the family $\{u_{R}^n\;:\; n\ge 1, R\ge1\}$ is monotone in $n$ and in $R$ and uniformly bounded on compact sets of $\R^D$. By elliptic regularity, letting $n\to+\infty$ and then $R\to+\infty$, we deduce that $u_{R}^n$ converges to a radial function $\super u$ solving \eqref{els} and such that $\super u\ge u$.
\hfill\qed

\begin{lemma}\label{propPhi}  Make the same assumptions as in Theorem \ref{existence} and suppose that $u\longmapsto f(u)/u$ is increasing in $I=[M,+\infty)$. Then, the function $\Phi$ defined by \eqref{fiddgr2} is decreasing in $I$ and $\lim_{+\infty}\Phi=0$. In particular, there exist $\varepsilon>0$ such that $\Phi:I\to(0,\varepsilon)$ is invertible.
\end{lemma}
\proof 
Let us first note that $f$ is increasing in $I$ and consider the change of variable $s=F(t)$. Then $t=F^{-1}(s)$ and
$$
\Phi(u)=\int_{F(u)}^{+\infty}\frac{(F^{-1})'(s)}{\sqrt{s-F(u)}}ds=\int_0^{+\infty}\frac{ds}{\sqrt{s}f(F^{-1}(s+F(u)))}.
$$
Thus $\Phi$ is decreasing  in $I$ and by monotone convergence, we have $\lim_{u\ri+\infty }\Phi(u)=0$.
\hfill\qed

\subsection{Proof of Theorem \ref{tunbounded}}
Let $u$ be an unbounded solution of \eq{els}. From \eqref{rhoc} we can find $c\in (0,1)$ and a positive nonincreasing function $\bar\rho$ such that
$\rho(x)\geq \bar{\rho}(|x|)$ in $\R^D$ and $\bar\rho(r)=cr^{-\alpha}$ for $r\geq 1$.
We next apply Lemma \ref{sandwich}  (for $\bar \rho$ instead of $\rho$) to deduce the existence of $v=v(|x|)$ such that $u\leq v(|x|)$ in $\RR^D$ and
$$
v''+\frac{D-1}{r}v'=\bar\rho(r) f(v),\quad r\geq 0.
$$
This implies that $r^{D-1}v'$ and $v$ are nondecreasing. Since $u$ is unbounded it follows that $v(r)\ri+\infty$ as $r\ri+\infty$.
Also we have
\begin{equation}\label{vD}
cr^{-\alpha}f(v)\leq v''+\frac{D-1}{r}v'\quad\mbox{ for all } r\geq 1.
\end{equation}
We next multiply with $rv'$ in \eq{vD} and integrate over $[r,s]$, where $1\leq r\leq s\leq 2r$. We obtain
$$
c\int_r^s t^{1-\alpha}f(v)v'dt\leq \frac{s}{2}v'^2(s)+\frac{2D-3}{2}\int_r^s v'^2(t)dt,
$$
and so
\begin{equation}\label{dd1}
\tilde{c}s^{1-\alpha}[F(v(s))-F(v(r))]\leq \frac{s}{2}v'^2(s)+\frac{2D-3}{2} \int_r^s v'^2(t)dt,
\end{equation}
for all $r\leq s\leq 2r$. Using the fact that $t\longmapsto t^{D-1}v'(t)$ is nondecreasing we have
$$
 \int_r^s v'^2(t)dt\leq [s^{D-1}v'(s)]^2\int_r^s t^{2-2D}dt\leq C(D) s v'^2(s),
$$
for all $1\leq r\leq s\leq 2r$. This last estimate combined with \eq{dd1} yields
$$
s^{-\alpha}[F(v(s))-F(v(r))]\leq  Cv'^2(s),
$$
for all $1\leq r\leq s\leq 2r$. Therefore
$$
cs^{-\alpha/2}\leq \frac{v'(s)}{\sqrt{ F(v(s))-F(v(r))  }},
$$
for all $1\leq r\leq s\leq 2r$. Integrating over $[r,2r]$ we find
$$
cr^{1-\alpha/2}\leq \int_{v(r)}^{v(2r)} \frac{dt}{\sqrt{ F(t)-F(v(r)) }}\leq \Phi(v(r))\quad\mbox{ for all }  r\geq 1.
$$
So, for $r$ large enough, we may use Lemma \ref{propPhi} and apply $\Phi^{-1}$ inverse to the above inequality. It follows that
\begin{equation}\label{dd2}
u(x)\leq v(r)\le \Gamma(r):=\Phi^{-1}(cr^{1-\alpha/2})\quad\mbox{ for all } x\in\partial B_r.
\end{equation}
Let us note that $u$ satisfies $\Delta u=a(x)u$ in $\RR^D$ where
$$
a(x)=\rho(x)\frac{f(u)}{u}.
$$
Using now \eq{dd2} together with \eq{rhoc} and \eq{fiddgr} we find
$$
a(x)\leq cr^{-\alpha}\frac{f(\Gamma(r))}{\Gamma(r)}\leq \frac{C}{r^2},
$$
for all $r>1$ large and $x\in \partial B_r$. We next make use of Harnack's inequality \cite[Theorem 8.2]{gt} to derive the existence of $C>0$ independent of $u$ such that for all $r>1$ large we have
$$
\sup_{\partial B_r} u\leq C\inf_{\partial B_r} u.
$$
Since $u$ is subharmonic and unbounded it follows that $u(x)\ri+\infty$ as $|x|\ri+\infty$. This finishes the proof of Theorem \ref{tunbounded}. \hfill\qed

\section{Uniqueness}




\subsection{Proof of Theorem \ref{uniq}}
Let $u$ be a radial ELS of \eq{els}. We set
\begin{equation}\label{LM}
v(t)=u(|x|)\,,\quad\mbox{  where } t=|x|^{1-\frac{\alpha}{2}}.
\end{equation}
Then, $v$ solves
\begin{equation}\label{vequation}
\left\{
\begin{aligned}
\frac{d^2v}{dt^2}+\frac{K}{t}\frac{dv}{dt}&= \frac{4}{(\alpha-2)^2}f(v(t)),&{\text{for $t\in(0,1],$}} \\
\lim_{t\to0^+} v(t)&=+\infty, &
\end{aligned}
\right.
\end{equation}
where $K:=\frac{\alpha-2D+2}{\alpha-2}\in(-\infty,1)$.  Note that $u=u(r)$ is a strictly increasing function of $r=\vert x\vert$. In particular, the mapping $v=v(t)$ is invertible. Let $t=t(v)$ denote its inverse mapping and let
$$V=-\frac{dv}{dt}(t(v)),$$
seen as a new function of the variable $v$. Up to replacing $f$ by $\frac{(\alpha-2)^2}{4}f$, \eqref{vequation} is equivalent to
\begin{equation}\label{new}
\left\{
\begin{aligned}
\displaystyle V\frac{dV}{dv}&-\frac{K}{t}V =f(v),&\text{for $v\in[a,+\infty),$} \\
\displaystyle t(v) &=\int_v^{+\infty}\frac{ds}{V(s)},&
\end{aligned}
\right.
\end{equation}
where  {$a=u(1)$}.

\noindent{\bf Step 1.} The mapping $v\to (t^KV)(v)$ is increasing.
Indeed,
$$
\frac{d}{dv}[t^KV] = -Kt^{K-1}+t^K\frac{dV}{dv}=t^K \left[\frac{dV}{dv}-\frac{K}{t}\right]
=t^K\frac{f(v)}{V}>0.
$$
\noindent{\bf Step 2.} Reduction to the radial case. Take two positive, radially symmetric and decreasing functions $\rho_{1},\rho_{2}$ and $R>0$ large such that
$$
\rho_{1}(x)=\rho_{2}(x)=\rho(x)=\vert x\vert^{-\alpha}\quad\mbox{ for }\vert x\vert>R
$$
 and $\rho_{2}\le \rho\le\rho_{1}$ in $\R^D$. Rescaling the space variable if necessary, we may always assume that $R=1$.
Now let $u_{1}$ be the minimal ELS to
\begin{equation*}
\Delta u=\rho_{1}(x) f(u)\,,\;\;
u>0 \qquad\text{in $\R^D$ $(D\ge3)$,}
\end{equation*}
given by Theorem \ref{existence}. Let $u_{2}$ be any ELS of \eq{els}. We want to prove that
$$
\lim_{\vert x\vert\to+\infty}[u_{2}(x)-u_{1}(x)]=0.
$$
By minimality, $u_{1}$ is radial and $u_{1}\le u_{2}$. We can also assume that $u_2$ is radial, otherwise by Lemma \ref{sandwich}  (for $\rho=\rho_2$), there exists a radial ELS  $\bar u_{2}$ of
\begin{equation*}
\Delta u=\rho_{2}(x) f(u)\,,\;\;
u>0 \qquad\text{in $\R^D$ $(D\ge3)$,}
\end{equation*}
such that $\bar u_{2}\ge u_{2}$ and we only need to replace $u_2$ by $\bar u_2$ in what follows.

\noindent Let $t_{i}, V_{i}$, $i=1,2$ denote the solutions to \eqref{new}
associated to $u_{1}$ and $u_{2}$ respectively. Then,

\noindent{\bf Step 3.} $V_{1}\ge V_{2}$ and $t_{1}\le t_{2}$ for $v$ sufficiently large.

Since $u_1\leq  u_2$, their inverse mappings satisfy $r_2\le r_{1}$, which implies $t_1\leq t_2$. Let us prove that $V_2\leq V_1$ for large $v$.
We argue by contraction, assuming there exists $\{u_k\}\ri +\infty$ such that $V_1(u_k)<V_2(u_k)$.
Since $t_1\leq t_2$ and $dt_{i}/dv = -1/V_{i}$, there exists another sequence $\{\tilde u_k\}\ri +\infty$ such that $V_2(\tilde u_k)\le V_1(\tilde u_k)$. So, $V_{1}-V_{2}$ changes sign infinitely many times. By the intermediate value theorem, $V_{1}-V_{2}$ vanishes infinitely many times. By the mean value theorem, we obtain at last an unbounded  sequence $\{w_{n}\}$ such that
$$\frac{d(V_1-V_2)}{du}(w_n)=0\quad\mbox{and}\quad sign(V_1(w_n)-V_2(w_n))=(-1)^{n}.$$
Using \eqref{new}, we have
$$0=\frac{d(V_1-V_2)}{du}(w_n)=\left(\frac{K}{t_1}-\frac{K}{t_2}\right)(w_n)-f(w_n)\left(\frac{1}{V_1}-\frac{1}{V_2}\right)(w_n).$$
The first term in the right-hand side has the sign of $K$, while the second term has the sign of $(-1)^n$, which is a contradiction.

At this stage, we need to distinguish the cases $K<0$ and $K\in[0,1)$. We begin with the latter.

\noindent{\bf Step 4a.} Assume $K\in[0,1)$. There exists a constant $C>0$ such that
$$
0\le t_{1}^KV_{1}-t_{2}^KV_{2}\le C\qquad\text{ for large $v$.}
$$
Let
$$
h= t_{1}^KV_{1}-t_{2}^KV_{2}.
$$
Then, using Step 3,
$$
\frac{dh}{dv}=f(v)\left[\frac{t_{1}^K}{V_{1}} - \frac{t_{2}^K}{V_{2}}\right]\le0.
$$
It follows that $h$ is bounded above and has constant sign for large $v$. Assume by contradiction that $h(v)<0$ for large $v$. Then, $t_{1}^KV_{1}< t_{2}^KV_{2}$, that is,
$$
-\frac{d}{dv}\left[\frac{t_{2}^{1-K}}{1-K}\right] <
-\frac{d}{dv}\left[\frac{t_{1}^{1-K}}{1-K}\right].
$$
Integrating between $v$ and $+\infty$ yields
$$
t_{2}^{1-K}< t_{1}^{1-K},
$$
contradicting $t_{1}\le t_{2}$, since $K<1$.

\noindent{\bf Step 5a.} If $K\in[0,1)$, there holds
\begin{equation} \label{esti crucial} 0\le  u_{2}-  u_{1}\le C\; t_{2}^{1-K}( u_{2}).
\end{equation}
\noindent Since
$$\int_{v_1}^{+\infty} \frac{ds}{t_1^KV_1}=\frac{t^{1-K}}{1-K}=\int_{v_{2}}^{+\infty} \frac{ds}{t_2^KV_2},$$
we have, using Step 1 on the one hand and Step 4a on the other,
\begin{equation*}
\begin{split}
\frac{v_2-v_1}{t_1^K(v_2)V_1(v_2)}\leq \int_{v_1}^{v_2}\frac{ds}{t_1^KV_1}&=\int_{v_2}^{+\infty} \frac{h(s)}{(t_1^KV_1)(t_2^KV_2)}ds\\
&\leq\frac{||h||_\infty}{t_1^K(v_2)V_1(v_2)}\int_{v_2}^{+\infty} \frac{ds}{t_2^KV_2}.
\end{split}
\end{equation*}
And so,
$$0\leq v_2(t)-v_1(t) \leq \| h \|_{\infty} \frac{t_2^{1-K}(v_2(t))}{1-K}\ri 0 \mbox{ as }t\ri 0^+.$$
Eq. \eqref{esti crucial} follows, which completes the proof of Theorem \ref{uniq} in the case $K\ge0$.\\

We turn to the case $K<0$. Let $w=V^2_1-V^2_2\geq 0$ and consider the function $E$, defined for $\lambda \in [0,1]$ by
$$E(\lambda)=\displaystyle{(-K)\frac{\sqrt{2(V_1^2-\lambda w)}}{\int_u^{+\infty} \frac{d\sigma}{\sqrt{2(V_1^2-\lambda w)}}}=\frac {-KW}{T}},$$
where $W=\sqrt{2(V_1^2-\lambda w)}$ and $T=\int_v^{+\infty} \frac{ds}{W}$.\\
\noindent{\bf Step 4b.} Assume $K<0$. For $\lambda\in [0,1]$,
$$\frac{dE}{d\lambda}=K\left[\frac{w}{WT}+\frac{W}{T^2}\int_v^{+\infty} \frac{w}{W^3}ds\right]$$
and $E$ is concave.\\
With
$$\frac{dW}{d\lambda}=-\frac{w}{W}\quad\mbox{and} \quad \frac{dT}{d\lambda}=\int_v^{+\infty} \frac{w}{W^3}ds,$$
we obtain easily the expression of the first derivative of $E$. The second derivative of $E$ is given by
\begin{equation}\label{D2E}
\begin{split}
\frac{d^2E}{d\lambda^2}=&{K}\left[\frac{w^2}{W^3T}+\frac{3W}{T^2}\int_v^{+\infty} \frac{w^2}{W^5}ds\right]\\
&-{K}\left[\frac{2w}{WT^2}\int_v^{+\infty} \frac{w}{W^3}ds+\frac{2W}{T^3}\left(\int_v^{+\infty} \frac{w}{W^3}ds\right)^2\right].
\end{split}
\end{equation}
By the Cauchy-Schwarz inequality, we have
$$\left(\int_v^{+\infty} \frac{w}{W^3}ds\right)^2\leq \left(\int_v^{+\infty} \frac{ds}{W}\right)\left(\int_v^{+\infty} \frac{w^2}{W^5}ds\right)=T\int_v^{+\infty} \frac{w^2}{W^5}ds.$$
Hence, the second term in the right-hand side of \eqref{D2E} is smaller than
$$-K\left[\frac{2w}{WT^{\frac32}}\left(\int_v^{+\infty} \frac{w^2}{W^5}ds\right)^{\frac12}+\frac{2W}{T^2}\int_v^{+\infty} \frac{w^2}{W^5}ds\right].$$
Plugging in \eqref{D2E}, we obtain
$$\frac{d^2E}{d\lambda^2}\leq{K}\left[\frac{w}{W^{\frac32}T^{\frac12}}-\frac{W^{\frac12}}{T}\left(\int_v^{+\infty} \frac{w^2}{W^5}ds\right)^\frac12\right]^2\leq 0.$$
\noindent{\bf Step 5b.} If $K<0$, there holds
\begin{equation}\label{estKn}
0\le  u_2-u_1\le C\,t_2( u_2),
\end{equation}
where $C$ is a positive constant.\\
By \eqref{new}, we have
\begin{equation}\label{(2)}
0=\frac{d}{dv}(V^2_1-V^2_2)-2K\left(\frac{V_1}{t_1}-\frac{V_2}{t_2}\right)=\frac{dw}{dv}+E(0)-E(1).\\
\end{equation}
Therefore,
$$
\frac{dw}{dv}=E(1)-E(0)\leq \frac{dE}{d\lambda}(0).
$$
By Step 4b and \eqref{(2)}, we deduce that
\begin{equation}\label{(2b)}
\frac{dw}{dv}-K\frac{w}{V_1t_1}\leq K\frac{V_1}{t_1^2}\int_v^{+\infty} \frac{w}{V_1^3}ds\leq 0.
\end{equation}
Let $q=\frac{w}{V_1}$. The derivative of $q$ is given by
\begin{equation*}
\frac{dq}{dv}=\frac{1}{V_1}\left(\frac{dw}{dv}-\frac{dV_1}{dv}\frac{w}{V_1}\right).
\end{equation*}
Since $V_1$ verifies \eqref{new}, we have
\begin{equation}\label{(3)}
\frac{dq}{dv}=\frac{1}{V_1}\left(\frac{dw}{dv}-K\frac{w}{V_1t_1}\right)-\frac{w}{V_1^3}f(v).
\end{equation}
Eqs. \eqref{(2b)}-\eqref{(3)} imply
$$\frac{dq}{dv}+q\frac{f}{V_1^2}\leq0.$$
Integrating the above inequality,
\begin{equation}\label{(4)}
q(v)\leq Ce^{-\displaystyle{\int_{v_0}^{v}\frac{f}{V_1^2}ds}}.
\end{equation}
Observe, using \eqref{new}, that the function $v\ri \frac{V^2_1}{2}-F(v)$ is decreasing. So, up to replacing $F$ by $\tilde F(v)=F(v)-F(v_0)+\frac{V_1^2(v_0)}{2}$, we have $V_1\leq \sqrt{2\tilde F}$ for $v\ge v_0$. Thus,
\begin{equation}\label{(6)}
q(v)\leq \frac{C}{\sqrt{\tilde F}}.
\end{equation}
Finally we proceed as in Step 5a. Since
$$\int_{v_1}^{+\infty} \frac{ds}{V_1}=t=\int_{v_2}^{+\infty} \frac{ds}{V_2},$$
we have
\begin{equation*}\label{fin}
\frac{v_2-v_1}{\sqrt{\tilde F(v_2)}}\leq  \int_{v_1}^{v_2} \frac{ds}{V_1}=\int_{v_2}^{+\infty} \frac{q\;ds}{V_2(V_1+V_2)}\leq  c\;q(v_2)\int_{v_2}^{+\infty} \frac{ds}{V_2}\leq \frac{c\;t_2}{\sqrt{\tilde F(v_2)}}.
\end{equation*}
Now \eqref{estKn} follows and this finishes the proof of Theorem \ref{uniq}.
\hfill\qed

\medskip

\subsection{Proof of Theorem \ref{uniqueness}}
As in the proof of Theorem \ref{uniq}, we may restrict ourselves to the radial case. Further, given a radial ELS $u$ to \eq{els} we  make the change of variable $t=\vert x\vert^{2-D}$, $v(t)=u(\vert x\vert)$. Then, $v$ solves
\begin{equation}\label{vequation2}
\left\{
\begin{aligned}
\frac{d^2v}{dt^2}&= \tilde\rho(t)f(v(t)),&\text{for $t\in(0,1],$} \\
\lim_{t\to0^+} v(t)&=+\infty, &
\end{aligned}
\right.
\end{equation}
where
$$
\tilde\rho(t)=\frac1{(D-2)^2}r^{2D-2}\rho(r),\qquad t=r^{2-D}.
$$
Letting, as in the proof of Theorem \ref{uniq}, $t=t(v)$ denote the inverse map of $v=v(t)$, and letting $V=-\frac{dv}{dt}(t(v))$, we arrive at the system
\begin{equation}\label{neweq}
\left\{
\begin{aligned}
\displaystyle V\frac{dV}{dv}& =\tilde\rho(t(v))f(v),&\text{for $v\in[a,+\infty),$} \\
\displaystyle t(v) &=\int_v^{+\infty}\frac{ds}{V(s)},&
\end{aligned}
\right.
\end{equation}
where $a=u(1)$.
Now take two radial ELS to \eq{els} $u_{i}$, $i=1,2$ and let $t_{i}, V_{i}$ denote the new unknowns associated to $u_{i}$.

\noindent{\bf Step 1.} $V=V_{i}$ satisfies
\begin{equation} \label{MariusHopital}
\lim_{v\to+\infty}\frac{V^2(v)}{F(v)}=\frac2{(D-2)^2}.
\end{equation}
Indeed, by \eqref{vequation2} and L'H\^opital's rule, we have
\begin{equation*}
\lim_{t\ri 0^+}\frac{\left(\frac{dv}{dt}\right)^2}{\frac{2}{(D-2)^2}F(v)}
=\lim_{t\ri 0^+}\frac{2\tilde{\rho}(t) f(v)}{\frac{2}{(D-2)^2}f(v)}=1,
\end{equation*}
where we used assumption \eqref{cond2}.

\noindent{\bf Step 2.} There exists a constant $C>0$ such that
\begin{equation} \label{intermediate}
\sqrt{F(v)} \vert V_{1}(v)-V_{2}(v)\vert \le C\left[
\int_{v_{0}}^v f(w)
\left(
\int_{w}^{+\infty}\frac{\vert V_{1}-V_{2}\vert}{F}ds
\right)dw +1 \right].
\end{equation}
To see this, take a large constant $v_{0}>0$ (to be fixed later on) and integrate \eqref{neweq} between $v_{0}$ and $v$:
$$
\frac12(V_{1}^2-V_{2}^2) = \int_{v_{0}}^v f(w)\left[\tilde\rho(t_{1}(w))-\tilde\rho(t_{2}(w))\right]\;dw+ c,
$$
where $c=\frac12(V_{1}^2-V_{2}^2)(v_{0})$.
Assumption \eqref{cond2}  implies that $\tilde\rho$ is Lipschitz continuous. Using this fact in the right-hand side of the above equation, and \eqref{MariusHopital} in the left-hand side, we deduce that
$$
\sqrt{F(v)} \vert V_{1}(v)-V_{2}(v)\vert \le C\left(
\int_{v_{0}}^v f(w) \vert t_{1}-t_{2}\vert\;dw +1
\right).
$$
Using the definition of $t_{i}$ and \eqref{MariusHopital} again,  we derive \eqref{intermediate}.

\noindent{\bf Step 3.} The following integral is convergent
\begin{equation} \label{intermediate2}
\int_{v_{0}}^{+\infty} \frac{\vert V_{1}-V_{2}\vert}{F}\left(
\int_{v_{0}}^{v\wedge s}f(w)\;dw
\right)\;ds.
\end{equation}
Indeed, by \eqref{MariusHopital} and $(KO)$, the integral
$$
\int_{w}^{+\infty}\frac{\vert V_{1}-V_{2}\vert}{F}\;ds
$$
is convergent.  Thus, so is the double integral
$$
\int_{v_{0}}^v f(w)\left(
\int_{w}^{+\infty}\frac{\vert V_{1}-V_{2}\vert}{F}\;ds
\right)\;dw.
$$
By Fubini's theorem, the integral in \eqref{intermediate2} is also convergent.

\noindent{\bf Step 4.} There exists  two constants $C,U_{0}>0$ such that for all $U\ge U_{0}$, and all $v\in(v_{0},U)$, we have
$$
\sqrt{F(v)}\vert V_{1}-V_{2}\vert \le C\left(
\int_{v_{0}}^{U}\vert V_{1}-V_{2}\vert\;ds +1
\right).
$$
By \eqref{intermediate} and Fubini's theorem,
\begin{equation} \label{intermediate3}
\sqrt{F(v)} \vert V_{1}(v)-V_{2}(v)\vert \le C\left[
\int_{v_{0}}^{+\infty} \frac{\vert V_{1}-V_{2}\vert}{F}\left(
\int_{v_{0}}^{v\wedge s}f(w)\;dw
\right)\;ds +1
\right].
\end{equation}
Also, by Step 3, there exists $U_{0}>0$ sufficiently large such that for all $U\ge U_{0}$,
$$
\int_{v_{0}}^{+\infty} \frac{\vert V_{1}-V_{2}\vert}{F}\left(
\int_{v_{0}}^{v\wedge s}f(w)\;dw
\right)\;ds\le 2
\int_{v_{0}}^{U} \frac{\vert V_{1}-V_{2}\vert}{F}\left(
\int_{v_{0}}^{v\wedge s}f(w)\;dw
\right)\;ds.
$$
Using this fact in \eqref{intermediate3} we find
\begin{align*}
\sqrt{F(v)} \vert V_{1}(v)-V_{2}(v)\vert &\le C\left[
\int_{v_{0}}^{U} \frac{\vert V_{1}-V_{2}\vert}{F}\left(
\int_{v_{0}}^{v\wedge s}f(w)\;dw
\right)\;ds +1
\right]\\
&\le C\left(\int_{v_{0}}^{U}\vert V_{1}-V_{2}\vert\;ds+1
\right).
\end{align*}

\noindent{\bf Step 5.} There exists a constant $C>0$  such that
\begin{equation} \label{step5}
\sqrt{F(v)} \vert V_{1}(v)-V_{2}(v)\vert\le C.
\end{equation}
Fix $\varepsilon>0$ and choose $v_{0}>0$  large enough such that
$$
\int_{v_{0}}^{+\infty}\frac{ds}{\sqrt F}<\varepsilon.
$$
By Step 4,
\begin{align*}
\sqrt{F(v)} \vert V_{1}(v)-V_{2}(v)\vert&\le C
\left(
\left\|(V_{1}-V_{2})\sqrt F\right\|_{L^\infty(v_{0},U)}
\int_{v_{0}}^{U}\frac{ds}{\sqrt F}+1
\right)\\
&\le C\varepsilon \left\|(V_{1}-V_{2})\sqrt F\right\|_{L^\infty(v_{0},U)} + C.
\end{align*}
This being true for all $v\in(v_{0},U)$, we deduce that
$$
(1-C\varepsilon)\left\|(V_{1}-V_{2})\sqrt F\right\|_{L^\infty(v_{0},U)}\le C
$$
 By taking $\varepsilon<1/(2C)$ and letting $U\to+\infty$, we obtain \eqref{step5}.

\noindent{\bf Step 6.} End of proof.
For fixed $t>0$, let $v_{1}=v_{1}(t)$ and $v_{2}=v_{2}(t)$. By \eq{neweq} we have the identity
\begin{equation}\label{ld1}
t=\int_{v_1}^{+\infty} \frac{ds}{V_1(s)}=\int_{v_2}^{+\infty} \frac{ds}{V_2(s)}.
\end{equation}
Assume without losing any generality that $v_1\leq v_2$. We infer from \eqref{ld1} that
$$
\int_{v_1}^{v_2} \frac{ds}{V_1(s)}=\int_{v_2}^{+\infty} \left(\frac{1}{V_2(s)}-\frac{1}{V_1(s)}\right)ds.
$$
Using \eqref{MariusHopital} and \eqref{step5},  we find
\begin{equation}\label{ld3}
\frac{v_2-v_1}{\sqrt{F(v_2)}}\leq C \int_{v_2}^{+\infty} \frac{|V_2(s)-V_1(s)|}{F(s)}ds\leq
\frac{C}{\sqrt{F(v_2)}} \int_{v_2}^{+\infty} \frac{1}{F(s)}ds.
\end{equation}
This yields $v_2(t)-v_1(t)=o(1)$ as $t\ri 0^+$, as desired.
\hfill\qed

\medskip

\subsection{Proof of Theorem \ref{ELSconvex}} Let $\tilde u$ be the minimal ELS solution of \eq{els} and let $u$ be any ELS solution of \eq{els}. By our assumptions we can find a sequence of smooth domains $\{\Omega_k\}$ such that
\begin{enumerate}
\item[(a)] $\Omega_k\subset\subset\Omega_{k+1}$ for all $k\geq 1$;

\item[(b)] $\tilde u\geq M$  in $\RR^D\setminus\Omega_1$ where $M>0$ is the constant from (ii);

\item[(c)] $\rho$ is constant on each $\partial\Omega_k$ and $\sqrt{\rho}$ is superharmonic in $\RR^D\setminus\Omega_1$;
\item[(d)] for each $k\geq 1$ the mean curvature $H_k$ of $\partial\Omega_k$ satisfies
\begin{equation}\label{mck}
2(D-1)H_k\geq\frac{|\nabla\rho|}{\rho}\mbox{ on }\partial\Omega_k.
\end{equation}
\end{enumerate}
For all $k\geq 2$ consider the problem
\begin{equation}\label{problk}
\left\{
\begin{aligned}
\Delta u_k&=\rho(x)f(u_k) &&\quad\mbox{ in }\Omega_k\setminus\Omega_1, \\
u_k&=\inf_{\partial\Omega_k}\tilde u&&\quad\mbox{ on }\partial\Omega_k,\\
u_k&=u &&\quad\mbox{ on }\partial\Omega_1.
\end{aligned}
\right.
\end{equation}
Then $u$ is a supersolution of \eq{problk} while for any $\beta<\inf_{\partial\Omega_k} \tilde u$ we have that $w_\beta$ defined by \eq{beta} is a subsolution of \eq{problk}. Hence \eq{problk} has a smooth solution $u_k$ satisfying
$$
w_\beta\leq u_k\leq u\quad\mbox{ in }\Omega_k\setminus\Omega_1\quad\mbox{ for all }k\geq 2.
$$
Furthermore, taking a subsequence if necessary, we can pass to the limit as $k\rightarrow +\infty$ in \eq{problk} to derive that $u_\infty:=\lim_{k\rightarrow+\infty}u_k$ satisfies
\begin{equation}\label{problinf}
\left\{
\begin{aligned}
\Delta u_\infty&=\rho(x)f(u_\infty) &&\quad\mbox{ in }\RR^D\setminus\Omega_1, \\
u_\infty(x)&\rightarrow+\infty&&\quad\mbox{ as }|x|\rightarrow +\infty,\\
u_\infty&\leq  u&&\quad\mbox{ in }\RR^D\setminus\Omega_1,\\
u_\infty&=u &&\quad\mbox{ on }\partial\Omega_1.
\end{aligned}
\right.
\end{equation}
We shall next divide our proof into three steps.

\medskip

\noindent{\bf Step 1.} There exists a positive constant $C>0$ such that
\begin{equation}\label{scaef}
\frac{|\nabla u_\infty|^2}{\rho(x)}-2(F(u_\infty)+C)\leq 0\quad\mbox{ in }\RR^D\setminus\Omega_1.
\end{equation}

We first apply Theorem C.1.  in Appendix C for $u=u_k$ on $\Omega=\Omega_k\setminus\Omega_1$. Thus, the function
$$
P_k:=\frac{|\nabla u_k|^2}{\rho(x)}-2F(u_k)
$$
achieves its maximum either on $\partial\Omega_1$ or at critical points of $u_k$.
By elliptic regularity, $\{u_k\}$ is uniformly bounded in $C^1(\Omega_{2}\setminus\Omega_{1})$, so there exists a positive constant $C>0$
which is independent of $k$ such that
$$
\| P_{k}\|_{L^\infty(\partial\Omega_{1})} \le 2C.
$$
It follows that
$$
\frac{|\nabla u_k|^2}{\rho(x)}-2(F(u_k)+C)\leq 0\quad\mbox{ in }\Omega_k\setminus\Omega_1,
$$
for all $k\geq 2$. Passing to the limit with $k\rightarrow +\infty$ in the above estimate we obtain \eq{scaef}.

\medskip

\noindent{\bf Step 2.} $u=u_\infty$ on $\RR^D\setminus\Omega_1$.

We already know (see \eq{problinf}) that $u_\infty\leq u$ in $\RR^D\setminus\Omega_1$. For the converse inequality let $C>0$ be the constant from \eq{scaef} and set
$$
v=\int_u^{+\infty}\frac{dt}{\sqrt{2(F(t)+C)}}\,,\quad
v_\infty=\int_{u_\infty}^{+\infty}\frac{dt}{\sqrt{2(F(t)+C)}}.
$$
Then $w:=v-v_\infty$ satisfies
\begin{equation}\label{eqw}
\begin{aligned}
-\Delta w=&\left\{\frac{f(u)}{\sqrt{2(F(u)+C)}}-\frac{f(u_\infty)}{\sqrt{2(F(u_\infty)+C)}}\right\}
(\rho(x)-|\nabla v_\infty|^2)\\
&+\frac{f(u)}{\sqrt{2(F(u)+C)}}(|\nabla v_\infty|^2-|\nabla v|^2)\quad\mbox{ in }\RR^D\setminus\Omega_1.
\end{aligned}
\end{equation}
Since $\sqrt F$ is convex on $[M,+\infty)$ it easily follows that $\frac{f}{\sqrt{2(F+C)}}$ is increasing on $[M,+\infty)$. Also by \eqref{scaef}, we have
$$
\begin{aligned}
\rho(x)-|\nabla v_\infty|^2&=\rho(x)-\frac{|\nabla u_\infty|^2}{2(F(u_\infty)+C)}\\&=-
\frac{\rho(x)}{2(F(u_\infty)+C)}\left\{\frac{|\nabla u_\infty|}{\rho(x)}-2(F(u_\infty)+C)\right\}\\
&\geq 0 \quad\mbox{ in }\RR^D\setminus\Omega_1.
\end{aligned}
$$
Thus, from \eq{eqw} we deduce
$$
-\Delta w\geq \frac{f(u)}{\sqrt{2(F(u)+C)}}(|\nabla v_\infty|^2-|\nabla v|^2)\quad\mbox{ in }\RR^D\setminus\Omega_1.
$$
Let now
$$
b(x):=\frac{f(u)}{\sqrt{2(F(u)+C)}}\nabla(v_\infty+v).
$$
Then $w$ satisfies
$$
\left\{
\begin{aligned}
-&\Delta w+b(x)w\geq 0  &&\quad\mbox{ in }\RR^D\setminus\Omega_1, \\
&w(x)\rightarrow 0&&\quad\mbox{ as }|x|\rightarrow \infty,\\
&w=0&&\quad\mbox{ on }\partial\Omega_1.
\end{aligned}
\right.
$$
By the maximum principle we derive $w\geq 0$ in $\RR^D\setminus\Omega_1$ so $u\leq u_\infty$ in
$\RR^D\setminus\Omega_1$, which finally yields $u\equiv u_\infty$ in $\RR^D\setminus\Omega_1$.

\medskip

\noindent{\bf Step 3.}  There exists a unique ELS of \eq{els}.

Let $\tilde u$ be the minimal ELS solution of \eq{els} and let $u$ be any ELS solution of \eq{els}. Also denote by $\tilde u_k$ and $u_k$ the solutions of \eq{problk} corresponding to $\tilde u$ and $u$ respectively. Then, for all $k\geq 2$,  $w_k:=u_k-\tilde u_k$ satisfies
$$
\left\{
\begin{aligned}
\Delta w_k&=\rho(x)[f(u_k)-f(\tilde u_k)] \ge0&&\quad\mbox{ in }\Omega_k\setminus\Omega_1,\\
w_k&=u-\tilde u\geq 0&&\quad\mbox{ on } \partial\Omega_1,\\
w_k&=0&&\quad\mbox{ on } \partial\Omega_k.
\end{aligned}
\right.
$$
By the maximum principle it follows that
$$
w_k\leq \sup_{\partial\Omega_1}(u-\tilde u)\quad\mbox{ in }\Omega_k\setminus\Omega_1,
$$
and the equality is achieved for some $\xi_k\in\partial\Omega_1$.
Passing to the limit with $k\rightarrow +\infty$ we find that $w_\infty=u_\infty-\tilde u_\infty$ satisfies
$$
w_\infty\leq \sup_{\partial\Omega_1}(u-\tilde u)\quad\mbox{ in }\RR^D\setminus \Omega_1,
$$
and the equality holds at some point $\xi\in \partial\Omega_1$. Since $w_{\infty}$ is subharmonic in $\Omega_{1}$, the above inequality also holds in $\Omega_1$. By the strong maximum principle we deduce $w_{\infty}\equiv w_{\infty}(\xi)=c\geq 0$. Thus $u\equiv\tilde u+c$ and using the fact that both $u$ and $\tilde u$ are ELS to \eq{els} we find $c=0$, that is, $u\equiv \tilde u$. This finishes our proof.
\hfill\qed

\medskip

\subsection{Proof of Theorem \ref{coro3}}
Let $u_{\infty}$ be the minimal ELS of \eq{els}. Since $\rho$ is radial, so is $u_{\infty}$.
Thus, $u_{\infty}$ satisfies
$$
(r^{D-1} u'_{\infty})'=r^{D-1}\rho(r)f(u_{\infty})\quad\mbox{ for all }r\geq 0.
$$
We multiply by $2r^{D-1} u'_{\infty}$ and integrate over $[R,r]$. We find
$$\begin{aligned}
r^{2D-2}(u'_{\infty})^2(r)-R^{2D-2}(u'_{\infty})^2(R)&=2\int_R^r t^{2D-2}\rho(t) f(u_{\infty})( u'_{\infty}) dt\\
&\leq 2r^{2D-2}\rho(r)F(u_{\infty}).
\end{aligned}
$$
Hence, letting $C_{R}=R^{2D-2}(u'_{\infty})^2(R)$,
$$
\frac{(u'_{\infty})^2}{\rho} \le 2F(u_{\infty}) + \frac{C_{R}}{r^{2D-2}\rho}\le 2(F(u_{\infty}) + C).
$$
That is, \eqref{scaef} holds in $\R^D\setminus B_{R}$.
Let now $u$ be an arbitrary ELS of \eq{els} and proceed as in Step 2 of Theorem \ref{ELSconvex}.
\hfill\qed.

\appendix

\section{Minimality Principle}

Basic to our analysis is the following result, the proof of which is a straightforward generalization of that in \cite[Section 2]{ddgr}.

\noindent{\bf Proposition A.1.} {\rm (Minimality Principle)} {\it
Let $\Omega$ be a smooth and bounded domain of $\mathbb{R}^D$, $f\in C^1(\R)$, $\rho\in C^1(\super\Omega)$ and $g\in
C(\partial\Omega)$. Assume there exists $\sub u, \super u\in C^2(\bom)$ such that $\sub u\le \super u$ in $\o$ and
\begin{equation}\label{solution}
\left\{
\begin{aligned}
\Delta \sub u&\ge \rho(x)f(\sub u)\;\;\text{(resp. $\Delta \super u\le \rho(x)f(\super u)$)} &&\quad\mbox{ in }\Omega, \\
\sub u&\le g \;\text{(resp. $\super u\ge g$ )}&&\quad\mbox{ on } \partial\Omega.
\end{aligned}
\right.
\end{equation}
Then, there exists a unique solution $u\in C^2(\bom)$ to
\begin{equation}\label{dirichlet}
\left\{
\begin{aligned}
\Delta  u&= \rho(x)f(u)&&\quad\mbox{ in }\Omega, \\
u& = g &&\quad\mbox{ on } \partial\Omega,
\end{aligned}
\right.
\end{equation}
such that $\sub u\le u$ and $u|_\omega \le \super v$ for any open subset
$\omega$ of $\om$ and any function $\super v\in C^2(\bar\omega)$
satisfying
\begin{equation}\label{localsupersolution}
\left\{
\begin{aligned}
\Delta  \super v &\le \rho(x)f(\super v) &&\quad\mbox{ in } \omega, \\
\super v&\ge \sub u &&\quad\mbox{ in } \omega,\\
\super v&\ge u &&\quad\mbox{ on } \partial\omega.
\end{aligned}
\right.
\end{equation}
We call $u$ the minimal solution to \eqref{dirichlet} relative to
$\sub u$.}

\section{On Poisson's equation}
We collect here some basic results on Poisson's equation, the proof of which can be found in Appendix A of \cite{bk}.

\noindent{\bf Lemma B.1.} {\rm (see \cite{bk})} {\it Let $D\ge3$, let $c_{D}\vert x\vert^{2-D}$ be the fundamental solution of the Laplace operator, and let $h\in L^\infty_{loc}(\R^D)$, $h\ge0$ a.e. There exists a bounded solution to
\begin{equation} \label{poisson}
-\Delta U = h\qquad\text{in $\R^D$}
\end{equation}
if and only if $u_{\star}:=c_{D}\vert x\vert^{2-D}\star h\in L^\infty(\R^D)$. Furthermore, $u_{\star}$ is the minimal positive solution to \eqref{poisson}. }

\noindent {\bf Lemma B.2.} {\rm (see \cite{bk})} {\it Make the same assumptions as above. Then,
$$
\liminf_{\vert x\vert\to+\infty}u_{\star}(x)=
\lim_{R\to+\infty} \fint_{\partial B_{R}(0)} u_{\star}\;d\sigma=0.
$$}

\section{Maximum values for functionals related to nonlinear Dirichlet problems}

The main result in this section is a reformulation of \cite[Theorems 1-2]{schaefer} which applies to our setting. For the reader's convenience we have included here a complete proof.

\noindent {\bf Theorem C.1.} {\it
Let $\Omega\subset\RR^D$ be a bounded domain with $C^3$ boundary and $u\in C^2(\overline\Omega)$ be such that
$$
\left\{
\begin{aligned}
\Delta  u&= \rho(x)f(u)&&\quad\mbox{ in }\Omega, \\
u& = c\geq 0 &&\quad\mbox{ on } \partial\Omega,
\end{aligned}\right.
$$
where
\begin{enumerate}
\item[(i)] $f\in C^1[0,\infty)$, $f\geq 0$;
\item[(ii)] $\rho>0$, $\rho\in C^2(\overline\Omega)$, $\left.\rho\right\vert_{\partial\Omega}$ is constant, and $\sqrt{\rho}$ is superharmonic in $\Omega$;
\end{enumerate}
Consider the functional
$$
P=\frac{|\nabla  u|^2}{\rho(x)}-2F( u)\,,
$$
and let $x_{0}$ be a maximum point of $P$. Then, either $x_{0}$ is a critical point of $u$ or $x_{0}\in\partial\Omega$ and
\begin{equation}\label{meanh}
2(D-1)H<\frac{|\nabla \rho|}{\rho}\quad\mbox{ at }x_{0},
\end{equation}
where $H$ is the mean curvature of $\partial\Omega$ computed at $x_{0}$.}
\begin{proof} We perform the proof along two steps.

\noindent{\bf Step 1.} $P$ achieves its maximum either at a critical point of $u$ or at a point on the boundary $\partial\Omega$.

It suffices to show that
$$
\Delta P+L\cdot \nabla P\geq 0\quad\mbox{ in }\Omega_0,
$$
for some smooth vector field $L$ defined in $\Omega_0:=\{x\in\Omega:\nabla u(x)\neq 0\}$.
Remark that for any $1\leq j\leq D$ we have
\begin{equation}\label{sch1}
P_j=\frac{2}{\rho}\sum_{i=1}^D u_iu_{ij}-\frac{|\nabla u|^2\rho_j}{\rho^2}-2f(u)u_j.
\end{equation}
So,
\begin{equation}\label{sch2}
\begin{aligned}
\Delta P=&\frac{2}{\rho}\sum_{i,j=1}^D u^2_{ij}+\frac{2}{\rho}\sum_{i,j=1}^D u_iu_{ijj}
-\frac{4}{\rho^2}\sum_{i,j=1}^D u_iu_{ij}\rho_j-\frac{|\nabla u|^2\Delta\rho}{\rho^2}\\
&+\frac{2|\nabla u|^2|\nabla\rho|^2}{\rho^3}-2f'(u)|\nabla u|^2-2f(u)\Delta u.
\end{aligned}
\end{equation}
Let $Q=(Q_1,Q_2,\dots, Q_D)$, where
$$
-Q_j:=\frac{2}{\rho}\sum_{i=1}^D u_iu_{ij}+\frac{|\nabla u|^2}{\rho^2}\rho_j+2f(u)u_j,\;1\leq j\leq D.
$$
From \eq{sch1} we have
$$
\frac{4}{\rho^2}\left(\sum_{i=1}^D u_iu_{ij}\right)^2=-P_jQ_j+
\left(\frac{|\nabla u|^2}{\rho^2}\rho_j+2f(u)u_j \right)^2.
$$
By the Cauchy-Schwarz inequality the above relation yields
$$
\frac{4|\nabla u|^2}{\rho^2}\sum_{i=1}^D u^2_{ij}\geq -P_jQ_j+
\frac{|\nabla u|^4}{\rho^4}\rho_j^2+4\frac{|\nabla u|^2f(u)}{\rho^2}\rho_ju_j +4f^2(u)u^2_j,
$$
and so
$$
\frac{2}{\rho}\sum_{i,j=1}^D u^2_{ij}\geq \sum_{j=1}^D \frac{-Q_j\rho}{2|\nabla u|^2}P_j+
\frac{|\nabla u|^2|\nabla \rho|^2}{2\rho^3}+2\frac{f(u)}{\rho}\nabla u\cdot\nabla\rho +2f^2(u)\rho.
$$
Using this last estimate in \eq{sch2} we find
\begin{equation}\label{sch3}
\begin{aligned}
\Delta P+T\cdot \nabla P\geq &|\nabla u|^2\left(\frac{|\nabla\rho|^2}{2\rho^3}-\frac{\Delta\rho}{\rho^2}\right)+\frac{2|\nabla u|^2|\nabla\rho|^2}{\rho^3}-2f'(u)|\nabla u|^2 \\
&+\frac{2}{\rho}\sum_{i,j=1}^D u_iu_{ijj}-\frac{4}{\rho^2}\sum_{i,j=1}^D u_iu_{ij}\rho_j+2\frac{f(u)}{\rho}\nabla u\cdot\nabla\rho,
\end{aligned}
\end{equation}
where $T=\frac{\rho}{2|\nabla u|^2} Q$.
Since $\Delta u=\rho(x)f(u)$, by differentiation we have
\begin{equation}\label{sch4}
\frac{2}{\rho}\sum_{i,j=1}^D u_iu_{ijj}=\frac{2f(u)}{\rho}\nabla u\cdot\nabla \rho+2f'(u)|\nabla u|^2.
\end{equation}
Also from \eq{sch1} we have
\begin{equation}\label{sch5}
\frac{4}{\rho^2}\sum_{i,j=1}^D u_iu_{ij}\rho_j=\frac{2}{\rho}\sum_{j=1}^D P_j\rho_j+\frac{2|\nabla u|^2|\nabla\rho|^2}{\rho^3}+\frac{4f(u)}{\rho} \nabla u\cdot\nabla\rho.
\end{equation}
Let now $L=T+\frac{2}{\rho}\nabla \rho$. Combining \eq{sch3}-\eq{sch5} we obtain
$$\Delta P+L\cdot \nabla P\geq \frac{2|\nabla u|^2}{\rho\sqrt{\rho}}(-\Delta\sqrt\rho)\geq 0\quad\mbox{ in }\Omega_0.
$$

\medskip

\noindent{\bf Step 2.} If $P$ achieves its maximum at $x_{0}\in\partial\Omega$, then \eqref{meanh} holds.

Since $\rho$ is constant on $\partial\Omega$ and $\rho\ge \left.\rho\right\vert_{\partial\Omega}$, the outer unit normal $n$ to $\partial\Omega$ is given by $n=-\nabla\rho/|\nabla\rho|$ and
$$
\frac{\partial\rho}{\partial n}=-|\nabla \rho|\quad\mbox{ on }\partial\Omega.
$$
Since $u$ is constant on $\partial\Omega$, $\left\vert\frac{\partial u}{\partial n}\right\vert=\vert\nabla u\vert$ on $\partial\Omega$ and so
\begin{equation}\label{cv6}
\begin{aligned}
\frac{\partial P}{\partial n}=&\frac{\partial u}{\partial n}\left(\frac{2}{\rho}\frac{\partial^2 u}{\partial n^2}-2f(u)\right)+\frac{\left|\frac{\partial u}{\partial n}\right|^2|\nabla\rho|}{\rho^2}\quad\mbox{ on }\partial\Omega.
\end{aligned}
\end{equation}
On the other hand, since $u$ is constant on $\partial\Omega$, we have
$$
\rho f(u)=\Delta u=\frac{\partial^2 u}{\partial n^2}+(D-1)H \frac{\partial u}{\partial n},
$$
that is,
  $$
\frac{2}{\rho}\frac{\partial^2 u}{\partial n^2}=2f(u)-2(D-1)H\frac{\frac{\partial u}{\partial n}}{\rho}\quad\mbox{ on }\partial\Omega.
$$
Using this last equality in \eq{cv6}, we find
\begin{align*}
\frac{\partial P}{\partial n}&-\frac{|\frac{\partial u}{\partial n}|^2}{\rho}\left(2(D-1)H-\frac{|\nabla \rho|}{\rho} \right)\\
&=-\frac{|\nabla u|^2}{\rho}\left(2(D-1)H-\frac{|\nabla \rho|}{\rho} \right)  \quad\mbox{ at }x_{0}.\\
\end{align*}
The Hopf maximum principle then implies \eqref{meanh}.
\end{proof}

\addcontentsline{toc}{section}{Acknowledgements}

\noindent{\bf Acknowledgements.}  The authors would like to thank the referee for pointing us the paper \cite{yz} and for pertinent comments which led to an improved version of our paper. This work has been supported by The French Embassy in Ireland and by the joint EGIDE-IRCSET Programme between France and Ireland "Boundary blow-up solutions for partial differential equations and systems".

\bibliographystyle{amsalpha}

\begin{bibdiv}
\begin{biblist}

\bib{bvg}{article}{
   author={Bidaut-Veron, M-F.},
   author={Grillot, P.},
   title={Asymptotic behaviour of the solutions of sublinear elliptic
   equations with a potential},
   journal={Appl. Anal.},
   volume={70},
   date={1999},
   number={3-4},
   pages={233--258}
   
   
}

\bib{bk}{article}{
   author={Brezis, H.},
   author={Kamin, S.},
   title={Sublinear elliptic equations in ${\bf R}^n$},
   journal={Manuscripta Math.},
   volume={74},
   date={1992},
   number={1},
   pages={87--106}
   
}

\bib{cli}{article}{
author={K-S. Cheng},
author={J-T. Li},
title={On the elliptic equations $\Delta u=K(x) u^\sigma$ and $\Delta u=K(x) e^{2u}$},
journal={Trans. Amer. Math. Soc.},
volume={304},
date={1987},
pages={639--668}
}

\bib{cheng-ni}{article}{
   author={Cheng, K-S.},
   author={Ni, W-M.},
   title={On the structure of the conformal scalar curvature equation on
   ${\bf R}^n$},
   journal={Indiana Univ. Math. J.},
   volume={41},
   date={1992},
   number={1},
   pages={261--278}
}


\bib{cd}{article}{
   author={Costin, O.},
   author={Dupaigne, L.},
   title={Boundary blow-up solutions in the unit ball: asymptotics,
   uniqueness and symmetry},
   journal={J. Differential Equations},
   volume={249},
   date={2010},
   number={4},
   pages={931--964}
}

\bib{ddgr}{article}{
   author={Dumont, S.},
   author={Dupaigne, L.},
   author={Goubet, O.},
   author={R{\u{a}}dulescu, V.},
   title={Back to the Keller-Osserman condition for boundary blow-up
   solutions},
   journal={Adv. Nonlinear Stud.},
   volume={7},
   date={2007},
   number={2},
   pages={271--298}
}

\bib{dynkin}{article}{
   author={Dynkin, E.B.},
   title={A probabilistic approach to one class of nonlinear differential
   equations},
   journal={Probab. Theory Related Fields},
   volume={89},
   date={1991},
   number={1},
   pages={89--115}
}


\bib{d1}{article}{
   author={Dupaigne, L.},
   title={Sym\'etrie : si, mais seulement si ?},
   journal={Symmetry for elliptic PDEs, AMS Contemporary Mathematics},
   volume={528},
   date={2010},
}


\bib{ep}{article}{
   author={Engl{\"a}nder, J.},
   author={Pinsky, R.G.},
   title={On the construction and support properties of measure-valued
   diffusions on $D\subseteq{\bf R}^d$ with spatially dependent
   branching},
   journal={Ann. Probab.},
   volume={27},
   date={1999},
   number={2},
   pages={684--730}
}


\bib{gnn}{article}{
   author={Gidas, B.},
   author={Ni, W-M.},
   author={Nirenberg, L.},
   title={Symmetry and related properties via the maximum principle},
   journal={Comm. Math. Phys.},
   volume={68},
   date={1979},
   number={3},
   pages={209--243}
}


\bib{gt}{book}{
author={Gilbarg, D.},
author={Trudinger, N.},
title={Elliptic Partial Differential Equations of Second Order},
series={Springer-Verlag, Berlin},
year={1993},
}

\bib{hebey}{article}{
   author={Hebey, E.},
   title={Variational methods and elliptic equations in Riemannian geometry},
   journal={Notes from lectures given at ICTP},
   date={2003},
   review={ http://www.u-cergy.fr/rech/pages/hebey/NotesICTP.pdf},
}


\bib{K}{article}{
   author={Keller, J.B.},
   title={On solutions of $\Delta u=f(u)$},
   journal={Comm. Pure Appl. Math.},
   volume={10},
   date={1957},
   pages={503--510}
}

\bib{lair}{article}{
   author={Lair, A.V.},
   title={A necessary and sufficient condition for existence of large
   solutions to semilinear elliptic equations},
   journal={J. Math. Anal. Appl.},
   volume={240},
   date={1999},
   number={1},
   pages={205--218}
}


\bib{legall}{article}{
   author={Le Gall, J-F.},
   title={Probabilistic approach to a class of semilinear partial
   differential equations},
   conference={
      title={Perspectives in nonlinear partial differential equations},
   },
   book={
      series={Contemp. Math.},
      volume={446},
      publisher={Amer. Math. Soc.},
      place={Providence, RI},
   },
   date={2007},
   pages={255--272}
}


\bib{li-ni}{article}{
   author={Li, Y.},
   author={Ni, W-M.},
   title={On conformal scalar curvature equations in ${\bf R}^n$},
   journal={Duke Math. J.},
   volume={57},
   date={1988},
   number={3},
   pages={895--924}
}

\bib{ni-li}{article}{
   author={Li, Y.},
   author={Ni, W-M.},
   title={Radial symmetry of positive solutions of nonlinear elliptic equations in $\R^n$},
   journal={Commun. Part. Di. Eq.},
   volume={18},
   date={1993},
   pages={1043-1054},
}


\bib{ni}{article}{
   author={Ni, W-M.},
   title={On the elliptic equation $\Delta u+K(x)u^{(n+2)/(n-2)}=0$, its
   generalizations, and applications in geometry},
   journal={Indiana Univ. Math. J.},
   volume={31},
   date={1982},
   number={4},
   pages={493--529}
}

\bib{olofsson}{article}{
   author={Olofsson, A.},
   title={Apriori estimates of Osserman-Keller type},
   journal={Differential Integral Equations},
   volume={16},
   date={2003},
   number={6},
   pages={737--756}
}


\bib{O}{article}{
   author={Osserman, R.},
   title={On the inequality $\Delta u\geq f(u)$},
   journal={Pacific J. Math.},
   volume={7},
   date={1957},
   pages={1641--1647}
}

\bib{ren}{article}{
   author={Ren, Y-X.},
   title={Support properties of super-Brownian motions with spatially
   dependent branching rate},
   journal={Stochastic Process. Appl.},
   volume={110},
   date={2004},
   number={1},
   pages={19--44}
}
	
\bib{schaefer}{article}{
   author={Schaefer, P.W.},
   author={Sperb, R.},
   title={A maximum principle for a class of functionals in nonlinear Dirichlet problems},
   journal={Ordinary and partial differential equations, Lecture Notes in Math.},
   volume={564},
   date={1976},
   pages={400-406}
}

\bib{taliaferro}{article}{
   author={Taliaferro, S.D.},
   title={Radial symmetry of large solutions of nonlinear elliptic
   equations},
   journal={Proc. Amer. Math. Soc.},
   volume={124},
   date={1996},   
   pages={1043-1054},
   number={2},
   pages={447--455}
}

\bib{taliaferro2}{article}{
   author={Taliaferro, S.D.},
   title={Are solutions of almost radial nonlinear elliptic equations almost
   radial?},
   journal={Comm. Partial Differential Equations},
   volume={20},
   date={1995},
   number={11-12},
   pages={2057--2092}
}


\bib{yamabe}{article}{
   author={Yamabe, H.},
   title={On a deformation of Riemannian structures on compact manifolds},
   journal={Osaka Math. J.},
   volume={12},
   date={1960},
   pages={21--37}
}


\bib{yz}{article}{
   author={Ye, D.},
   author={Zhou, F.},
   title={Invariant criteria for existence of bounded positive solutions},
   journal={Discrete Contin. Dyn. Syst.},
   volume={12},
   date={2005},
   number={3},
   pages={413--424}
}



\bib{dong}{article}{
   author={Ye, D.},
   author={Zhou, F.},
   title={Existence and nonexistence of entire large solutions for some
   semilinear elliptic equations},
   journal={J. Partial Differential Equations},
   volume={21},
   date={2008},
   number={3},
   pages={253--262}
}





\end{biblist}
\end{bibdiv}

\end{document}